\documentclass[14pt]{amsart}

\usepackage{todonotes}

\usepackage{amsfonts, amssymb, amscd}

\usepackage[symbol]{footmisc}
\usepackage{amsmath}
\usepackage{bm}
\usepackage{verbatim}
\usepackage{mathrsfs}
\usepackage{lineno}
\usepackage{graphicx}
\usepackage{tikz-cd}
\usepackage{subcaption}
\usepackage{listings}
\usepackage{subfiles}
\usepackage[toc,page]{appendix}
\usepackage{mathtools}
\usepackage{comment}
\usepackage{enumerate}
\usepackage{enumitem}
\usepackage[all]{xy}

\usepackage{graphicx}
\graphicspath{{images/}}

\usepackage{appendix}
\usepackage{hyperref}
\hypersetup{
	colorlinks=true,
	citecolor=red,
	linkcolor=blue,
	filecolor=magenta,      
	urlcolor=red,
}
\newcommand{\Co}{\mathcal{O}}

\newcommand{\Supp}{\mathrm{Supp}}
\newcommand{\vol}{\mathrm{vol}}

\theoremstyle{plain} 
\newtheorem{thm}{Theorem}[section] 

\theoremstyle{definition} 

\theoremstyle{remark} 

\begin{document}
	\title{The volume function is upper semicontinuous on families of divisors}
	
	\author{Junpeng Jiao}
	\email{jiao$\_$jp@tsinghua.edu.cn}
	\address{Yau Mathematical Sciences Center, Tsinghua University, Beijing, China}
	

	\begin{abstract}
		We study the behavior of volumes of divisors in a family. We show that the volume of a divisor on the generic fiber equals the infimum of its volumes on fibers over any dense subset of the base. As an application, we show that the volume function of a divisor is upper semicontinuous in flat families with reduced and irreducible fibers.
	\end{abstract}
	
	\maketitle
	Throughout this paper, we work over algebraically closed number fields of characteristic zero. 
	\section{Introduction}
	In algebraic geometry, divisors play a central role in encoding geometric and arithmetic properties of varieties. From defining embeddings into projective spaces to controlling singularities in birational geometry, divisors are indispensable tools. Among the many invariants associated with divisors, the volume stands out as a measure of their asymptotic ``size" or ``complexity." Introduced as a higher-dimensional analog of the degree of a line bundle on a curve, the volume of a divisor captures the asymptotic growth rate of the dimension of its space of global sections. This invariant has become a cornerstone in modern birational geometry, with deep connections to the minimal model program, Fujita-type approximation theorems, and the study of Okounkov bodies. 
	
	The volume of a Cartier divisor $D$ on a $d$-dimensional projective variety $X$ over a field $K$ is defined as
	$$\vol(D)=\begin{cases}
		\limsup\limits_{n\rightarrow +\infty} \frac{\mathrm{dim}_KH^0(X,\Co_X(nD))}{n^d/d!},&\text{ if $D$ is pseudo-effective,}\\
		-\infty, &\text{ if $D$ is not.}
	\end{cases}$$
	Unlike traditional invariants such as the Euler characteristic or Hilbert polynomial coefficients, the volume is not a constructible function in flat families: there exist flat families of divisors which are big on infinitely many fibers but not on the generic fiber, see \cite{PS16} for an example. Generally, for a flat family of divisors, the volume on the generic fiber equals the volume on very general fibers but may jump up on fibers over a countable union of locally closed subsets of the base. 
	
	In this paper, we prove that the volume on the generic fiber actually equals the infimum of volumes on fibers over any Zariski dense subset of the base. Note this infimum may not be achieved (i.e., it may not be a minimum) if this dense subset does not contain any very general point.
	
	\begin{thm}\label{Main theorem: densely semicontinuous of volumes}
		Let $f: X\rightarrow T$ be a projective flat morphism of varieties and $D$ a Cartier divisor on $X$. Suppose $X_\eta$ is the generic fiber of $f$. Write $D_\eta:=D|_{X_\eta}$ and $ D_t:=D|_{X_t}$ for every closed point $t\in T$, then
		$$\vol(D_\eta)=\inf_{t\in T'} \vol(D_t)$$
		for any dense subset $T'\subset T$.
	\end{thm}
	
	As an application, we show that the volume function is actually upper semicontinuous for a flat morphism with reduced and irreducible fibers.
	\begin{thm}\label{upper semicontinuous of volume}
		Let $f:X\rightarrow T$ be a projective flat morphism of varieties with reduced and irreducible fibers and $D$ a Cartier divisor on $X$. Then for any $a\in \mathbb{R}$, the set
		$$\{t\in T\mid \vol(D_t)<a\}$$
		is open in $T$. 
	\end{thm}

	\section{Proof of Main Theorems}
	\begin{proof}[{Proof of Theorem \ref{Main theorem: densely semicontinuous of volumes}}]
		By the upper semicontinuity of cohomology, we have
		$$\vol(D_\eta)\leq \inf _{t\in T'} \vol(D_t).$$
		Suppose the result holds when $D_\eta$ is big. If $D_\eta$ is not big, we choose an ample divisor $L$ on $X$. Define
		$$\rho:=\inf\{\gamma\in \mathbb{R}_+\mid D_\eta+\gamma L_\eta \text{ is big}\}.$$
		By applying the result on $D+\gamma L$ for every $\gamma >\rho$, we have
		$$\vol(D_\eta+\gamma L_\eta)=\inf _{t\in T'} \vol(D_t+\gamma L_t).$$
		Consider the following two cases.
		
		Case 1: If $D_\eta$ is pseudo-effective, then $\rho=0$. By letting $\gamma\rightarrow 0+$, we have
		$$\vol(D_\eta)=\inf _{t\in T'} \vol(D_t)=0.$$
		
		Case 2: If $D_\eta$ is not pseudo-effective, then $\rho>0$. By letting $\gamma\rightarrow \rho+$, we have $\vol(D_\eta+\rho L_\eta)=\inf _{t\in T'} \vol(D_t+\rho L_t)=0$. Since $L$ is ample over $T$, $\vol(\gamma L_t)$ is independent of $t$, then there exists some $t\in T'$ such that $ \vol(D_t+\gamma L_t)<\vol(\gamma L_t)$, which means $D_t$ is not pseudo-effective. Then
		$$\vol(D_\eta)=\inf _{t\in T'} \vol(D_t)=-\infty.$$
	
		From now we may assume $D_\eta$ is big.
		We fix a sufficiently ample divisor $H$ on $X$ such that $H\pm D$ are ample over $T$.
		
		Suppose $S$ is a variety over $\mathbb{C}$ and $\pi:S\rightarrow T$ is a finite cover. Let $X_S:=X\times_ T S$ and $D_S$ be the pullback of $D$ on $X_S$. Let $S':=\pi^{-1}(T')$, then $S'$ is dense in $S$ since $T'$ is dense in $T$. Let $X_{\eta_S}$ be the generic fiber of $X_S\rightarrow S$, then it is easy to see that 
		$$\vol(D_\eta)=\vol(D_S|_{X_{\eta_S}})\text{ and }\vol(D_t)=\vol(D_S|_{X_{S,s}})$$
		for every closed point $t\in T$ and $s\in \pi^{-1}(t)$. Thus
		$$\vol(D_\eta)=\inf_{t\in T'} (D_t)\text{ if and only if }\vol(D_{\eta_S})=\inf_{s\in S'}\vol(D_S|_{X_{S,s}}).$$
		So to prove the result we are free to pass to a finite cover of $T$.

		Let $K:=k(T)$ and $\overline{K}$ be its algebraic closure. Let $X_{\bar{\eta}}$ be the base change of $X_\eta$ to $\overline{K}$ and $D_{\bar{\eta}}$ the pullback of $D_\eta$ on $X_{\bar{\eta}}$.
		For any $\epsilon\in \mathbb{R}_+$, let $g_{\bar{\eta}}:Y_{\bar{\eta}}\rightarrow X_{\bar{\eta}}$ be a Fujita approximation of $D_{\bar{\eta}}$ such that $g_{\bar{\eta}}^*D_{\bar{\eta}}\sim_{\mathbb{Q}} A_{\bar{\eta}}+E_{\bar{\eta}}$, where $Y_{\bar{\eta}}$ is smooth, $E_{\bar{\eta}}\geq 0$, $A_{\bar{\eta}}$ is ample, and 
		$$ \vol(D_{{\bar{\eta}}})\leq \vol(A_{\bar{\eta}})+\epsilon.$$
		
		Since the morphism $g_{\bar{\eta}}$ and $\mathbb{Q}$-divisors $A_{\bar{\eta}},E_{\bar{\eta}}$ are determined by finitely many equations in $\overline{K}$, they are defined over a finite extension of $K$. Therefore, after taking the corresponding finite base change and replacing $T$ and $\eta$, we can assume that $g_{\bar{\eta}}$ induces a birational morphism $g_\eta:Y_\eta\rightarrow X_\eta$ and $A_{\bar{\eta}},E_{\bar{\eta}}$ are the pullback of $\mathbb{Q}$-divisors $A_\eta,E_\eta$ on $Y_\eta$. It is easy to see that $Y_\eta$ is smooth, $E_\eta\geq 0$, and $A_\eta$ is ample. 
		
		Let $Y$ be a smooth compactification of $Y_\eta$ over $T$, $A$ and $E$ the closures of $A_\eta$ and $E_\eta$ in $Y$. Since $A_\eta$ is ample and $E_\eta$ is effective, after shrinking $T$, we may assume that 
		\begin{itemize}
			\item $A$ is ample over $T$,
			\item $\Supp(E)$ is flat over $T$,
			\item $X\rightarrow T,Y\rightarrow T$ are flat with reduced and irreducible fibers,
			\item $Y_\eta\rightarrow X_\eta$ induces a fiberwise birational morphism $g:Y\rightarrow X$ over $T$, and
			\item $g^*D\sim_{\mathbb{Q}} A+E$.
		\end{itemize}
		
		We assume $d:=\mathrm{dim}(X/Z)$.
		Since $H$ is sufficiently ample such that $H\pm D$ are ample over $T$, then $H_{\bar{\eta}}\pm D_{\bar{\eta}}$ are ample. By \cite[Theorem 11.4.21]{Laz04}, there exists a universal number $C\in \mathbb{R}_+$ such that 
		$$(A_{\bar{\eta}}^{d-1}\cdot E_{\bar{\eta}})^2 \leq C\cdot (H_{\bar{\eta}}^d) \cdot (\vol(D_{\bar{\eta}})-\vol(A_{\bar{\eta}}))\leq C\epsilon \cdot (H_{\bar{\eta}}^d).$$
		Let $C':=\sqrt{C\cdot H_{\bar{\eta}}^d}$, then $C'$ depends only on $H$ and $A_{\bar{\eta}}^{d-1}\cdot E_{\bar{\eta}}\leq C'\sqrt{\epsilon}$.
		
		Because $A_\eta^{d-1}\cdot E_\eta= A_{\bar{\eta}}^{d-1}\cdot E_{\bar{\eta}}$, then
		$A_\eta^{d-1}\cdot E_\eta\leq C' \sqrt{\epsilon}$ and
		$$A_\eta^{d-1}\cdot (A_\eta+E_\eta)\leq A_\eta^d+ C' \sqrt{\epsilon}.$$
		Because intersection numbers are locally constant in a flat family, then 
		\begin{equation}\label{Equation 1}
			A_t^{d-1}\cdot (A_t+E_t)\leq A_t^d +C' \sqrt{\epsilon}
		\end{equation}
		for every closed point $t\in T$.
		
		Note $A_t+E_t\sim_{\mathbb{Q}} g_t^*D_t$ for every $t\in T$. From now we fix a closed point $t\in T'$. By the Fujita approximation, for any $\epsilon'\in \mathbb{R}_+$ there exists a projective birational morphism $h_t:W_t\rightarrow Y_t$ such that $h_t^*g_t^*D_t\sim_{\mathbb{Q}} h_t^*(A_t+E_t)\sim_{\mathbb{Q}} A'_t+E'_t$, where $E'_t\geq 0$, $A'_t$ is ample, and $\vol(A'_t)\geq \vol(A_t+E_t)-\epsilon'=\vol(D_t)-\epsilon'$. 
		By the projection formula, we have
		\begin{equation}\label{Equation 2}
			(h_t^*A_t)^{d-1}\cdot A'_t\leq (h_t^*A_t)^{d-1}\cdot (A'_t+E'_t)= A_t^{d-1}\cdot (A_t+E_t).
		\end{equation}
		By the Generalized Inequality of Hodge type, see \cite[Theorem 1.6.1]{Laz04}, we have
		\begin{equation}\label{Equation 3}
			(A_t^d)^{\frac{d-1}{d}}(A_t^{'d})^{\frac{1}{d}}\leq (h_t^*A_t)^{d-1}\cdot A'_t.
		\end{equation}
		
		Now combine Equation \eqref{Equation 1}, \eqref{Equation 2}, and \eqref{Equation 3}, we have
		$$(A_t^d)^{\frac{d-1}{d}}(A_t^{'d})^{\frac{1}{d}}\leq (h_t^*A_t)^{d-1}\cdot A'_t \leq A_t^{d-1}\cdot (A_t+E_t) \leq  A_t^d +C' \sqrt{\epsilon}.$$
		Thus
		$$A_t^{'d}\leq ((A_t^d)^{\frac{1}{d}}+\frac{C'\sqrt{\epsilon}}{(A_t^d)^{\frac{d-1}{d}}})^d.$$
		
		Because $\vol(D_\eta)\geq A_\eta^d =\vol(A_\eta)\geq  \vol(D_\eta)-\epsilon$, $\vol(A_\eta)=\vol(A_t)$ for every $t\in T$ and $A_t^{'d}\geq \vol(D_t)-\epsilon'$, then
		\begin{equation}\label{inequality of volume}
			\vol(D_t)\leq ((\vol(D_\eta))^{\frac{1}{d}}+\frac{C'\sqrt{\epsilon}}{(\vol(D_\eta)-\epsilon)^{\frac{d-1}{d}}})^d+\epsilon'.
		\end{equation}
		
		Since $T'\subset T$ is dense, for any $\epsilon''\in \mathbb{R}_+$, by choosing $\epsilon$ and $\epsilon'$ small enough in Equation \eqref{inequality of volume} and shrinking $T$ accordingly, there always exists $t\in T'$ such that
		$\vol(D_t)\leq  \vol(D_\eta)+\epsilon''.$ Then we have
		$$\vol(D_\eta)=\inf _{t\in T'} \vol(D_t).$$
	\end{proof}
	
	\begin{proof}[{Proof of Theorem \ref{upper semicontinuous of volume}}]
		We may assume $a\in \mathbb{R}_+$.
		Fix $t_0\in T$ such that $\vol(D_{t_0})<a$, it suffices to prove that there exists an open neighborhood $U$ of $t_0$ such that 
		$$\vol(D_t)<a$$ 
		for any $t\in U$.
		
		Let $T':=\{t\in T\mid \vol(D_t)\geq a\}$ and $\overline{T'}$ be its closure, it suffices to show that $t_0\not \in \overline{T'}$. 
		
		Suppose there exists an irreducible component $S$ of $\overline{T'}$ such that $t_0\in S$. Let $\eta$ be the generic point of $S$, then according to the upper semi-continuity of cohomology we have
		$$a>\vol(D_{t_0})\geq \vol(D_\eta).$$
		
		Define $X_S:=X\times_T S$. Since $X\rightarrow T$ is projective, flat, with reduced and irreducible fibers, then $X_S\rightarrow S$ is also projective, flat, with reduced and irreducible fibers. Because by assumption $T'\cap S$ is dense in $S$, then by Theorem \ref{Main theorem: densely semicontinuous of volumes}, we have
		$$\vol(D_\eta)=\inf_{t\in T'\cap S} \vol(D_t)\geq a,$$
		which contradicts with that $a> \vol(D_\eta)$. Thus $t_0\not \in \overline{T'}$.
	\end{proof}

		\noindent\textbf{Acknowledgments}. The author would like to thank his advisor Christopher Hacon for many useful suggestions, discussions, and his generosity. He would also like to thank his postdoc mentor Caucher Birkar for his encouragement and constant support. The author also acknowledges Lingyao Xie, Bingyi Chen, Minzhe Zhu, and Xiaowei Jiang for their valuable comments. This work was supported by BMSTC and ACZSP. 	
	
\end{document}